\documentclass[12pt]{amsart}

\usepackage{color,graphicx,amssymb,latexsym,amsfonts,txfonts,amsmath,amsthm}
\usepackage{wrapfig}
\usepackage{epsfig}

\usepackage{psfrag}

\usepackage{amssymb}
\usepackage{graphicx}
\usepackage{pb-diagram}
\newtheorem{thm}{Theorem}[section]

\newtheorem{lemma}{Lemma}[section]
\newtheorem{prop}{Proposition}[section]

\begin{document}

\title[Boundary of the pyramidal  equisymmetric locus of ${\mathcal M}_n$.]{Boundary of the pyramidal equisymmetric locus of ${\mathcal M}_n$.}
\author{Raquel D\'{\i}az }
\address{Departamento de Geometr\'{\i}a y Topolog\'{\i}a. Facultad de Ciencias Matem\'aticas. Universidad Complutense de Madrid. Espa\~na. Phone n. +34 913944665.}
\email{radiaz@mat.ucm.es}

\author{V\'ictor Gonz\'alez-Aguilera}
\address{Departamento de Matem\'atica. Universidad T\'ecnica Federico Santa Mar\'{\i}a. Valpara\'{\i}so, Chile}
\email{victor.gonzalez@usm.cl}
\thanks{
  \mbox{\hspace{11pt}}{\it Key words}: Moduli space, stratification, Dehn-Thurston coordinates.\\
  \mbox{\hspace{11pt}}
The first author was partially supported by the Project MTM2012-31973 and by Project MTM2017-89420-P. The second author was partially supported by Project   PIA, ACT 1415}

\date{\today}

\begin{abstract} The augmented moduli space $\widehat {\mathcal M}_n$ is a compactification of moduli space $\mathcal M_n$ obtained by adding stable hyperbolic surfaces. The different topological types of the added stable surfaces produces a stratification of $\partial \widehat {\mathcal M}_n$.
Let  $\mathcal{P}_n\subset {\mathcal M}_n$ be the {\it pyramidal locus} in moduli space, i.e., the set of hyperbolic surfaces of genus $n$ such that the topological action of its preserving-orientation isometry group is the pyramidal action of the dihedral group $D_n$.  
The purpose of this paper is to state the complete list of   strata in the boundary of $\mathcal{P}_n$.

\end{abstract}

\maketitle

\maketitle
\section{Introduction}
%%%%%%%%%%%%%

Let $\mathcal{M}_{n}$ be the moduli space of   genus $n\geq2$, i.e., the space of hyperbolic surfaces of genus $n$   up to isometry.
The {\it augmented moduli space} $\widehat {\mathcal M}_n$, introduced by Abikoff (\cite{Abikoff}), is a compactification of  ${\mathcal M}_n$ by adding   stable hyperbolic  surfaces.   The topological type of the stable surfaces added  gives a stratification of $\widehat {\mathcal M}_n$. Since   topological    stable surfaces are codified by their dual stable graphs (see Section \ref{sec:HypSurfWithNodes} for the definitions), this gives a bijection 
between the set of  strata of $\widehat{\mathcal M}_n$ and  the  isomorphism classes of stable graphs   of   genus  $n$. 
 In this stratification, ${\mathcal M}_n$ is also a  stratum, the one  corresponding to the stable graph consisting on a unique vertex with weight $n$ and no edges. 
In the sequel, we will refer to these strata as the {\it topological strata } of $\widehat{\mathcal M}_n$. As notation, given a stable graph ${\mathcal G}$ of   genus $n$, the stratum corresponding to this graph is denoted by $\mathfrak{S}({{\mathcal G}})$.
 
 On the other hand, in \cite{Broughton}, Broughton analyzed   another stratification of moduli space, given by the topological action of the preserving-orientation isometry group of the points of ${\mathcal M}_n$. That is, two hyperbolic surfaces $X,X'$ are in the same {\it equisymmetric stratum (or locus)} if the  
action of ${\rm Iso}^+(X)$ on $X$ is topologically equivalent to the action of  ${\rm Iso}^+(X')$ on $X'$. 
In order to  give the topological class of an action of a group $G$ on a surface $S_n$, it suffices to give an epimorphism $\Phi\colon  \pi_1({\mathcal O})\to G$, where ${\mathcal O}$ is an orbifold and the kernel of $\Phi$ is isomorphic to $\pi_1(S_n)$. Thus, an epimorphism as above determines an equisymmetric locus of ${\mathcal M}_n$, which will be denoted $\mathcal{L}_n(\Phi,{\mathcal O},G)$.

 Let us fix an equisymmetric locus ${\mathcal L}=\mathcal{L}_n(\Phi,{\mathcal O},G)$ of ${\mathcal M}_n$. By $\partial {\mathcal L}$ we mean the boundary of ${\mathcal L}$ in $\widehat {\mathcal M}_n$. 
If  $ \mathfrak{S}({\mathcal G})$ is a topological  stratum of $\partial{\mathcal M}_n=\widehat {\mathcal M}_n\setminus {\mathcal M}_n$  such that $\mathfrak{S}({\mathcal G})\cap \partial {\mathcal L}{\mathbb N}ot=\emptyset$, we refer to this intersection as a {\it (topological) stratum} of $\partial{\mathcal L}$, and we say that the graph ${\mathcal G}$ is {\it realizable} in $ \partial {\mathcal L}$.
We are interested in the problem of determining which stable graphs of genus $n$ are realizable in the boundary of  a given equisymmetric locus ${\mathcal L}$. 
  It is easy to see that each  {\it   admissible} multicurve  $\Sigma$  on ${\mathcal O}$ determines a stratum $\mathfrak{S}_{\Sigma}$ in $\partial {\mathcal L}$ in the following way. By an      admissible multicurve  we mean  a family $\Sigma$ of simple non-homotopic closed curves or arcs joining cone-points of order 2 such that each component of ${\mathcal O}\setminus \Sigma$ has negative Euler characteristic. Then, one can consider a sequence of  hyperbolic orbifolds ${\mathcal O}_k$ homeomorphic to  ${\mathcal O}$ so that the length of the multicurve $\Sigma$ tends  to zero. By taking the covering surfaces, we obtain a sequence of surfaces $S_k$ in 
the equisymmetric locus ${\mathcal L}$, with limit a point in  a stratum that we denote by  $\mathfrak{S}_{\Sigma}$.  Moreover, all strata in $\partial {\mathcal L}$ are obtained in this way. The graph ${\mathcal D}_{\tilde\Sigma}$ associated to the stratum   $\mathfrak{S}_{\Sigma}$ is the   graph dual to the stable surface obtained by collapsing the multicurve  $\tilde\Sigma=p^{-1}(\Sigma)$ of $S$ (where $p\colon S\to {\mathcal O}$ is the branched covering map). 

    The main result of  our  paper \cite{Di-Go} (see Theorem \ref{thm:Combinatorial})
 describes  the graph ${\mathcal D}_{\tilde\Sigma}$ (vertices, edges, incidences, weights) for any given $\Sigma$.   
We next applied 
 this result to the locus $\mathcal{P}_n$ corresponding to the pyramidal action of the dihedral group $D_n$ on a surface of genus $n$. In particular, we found four families of stable graphs whose union contains all the graphs realizable in $\partial{\mathcal P}_n$.
  For three of these families, $ {\mathcal G}^1_{n,m}, \,{\mathcal G}^2_{n,k} ,\, {\mathcal G}^3_{n,m}$, we also proved that  they are actually realizable (see Theorem \ref{thm:pyramids}). 
    In this paper we prove that  the fourth family, ${\mathcal G}^4_{n,m,d}$ is also realizable (Theorem \ref{thm:Grafo4}). In this way, we obtain a complete explicit description of $\partial {\mathcal P}_n$.

    The strategy for the proof of Theorem \ref{thm:Grafo4} consists on  finding the convenient multicurves  $\Sigma$ in ${\mathcal O}$ such that ${\mathcal D}_{\tilde \Sigma}={\mathcal G}^4_{n,m,d}$. We will use Dehn-Thurston coordinates in the search of these multicurves.

 In Section  \ref{sec:Prelim} we recall some definitions and  the part of  Dehn-Thurston coordinates needed in this paper. In Section \ref{sec:MainThmInDiGo} we state  the main theorem in \cite{Di-Go}.
 Theorems \ref{thm:pyramids} and \ref{thm:Grafo4}  in Section \ref{sec:PyaramidalAction}   state the complete description of the strata in the  boundary of the pyramidal locus, the first one was   proved in \cite{Di-Go}. In Section \ref{sec:proof} we prove Theorem \ref{thm:Grafo4}.

 \section{Preliminaries.}\label{sec:Prelim}
 
\subsection{Stable hyperbolic surfaces and stable graphs}\label{sec:HypSurfWithNodes}

 Let $S$ be a  closed orientable  surface  of genus
$n$, and let ${\mathcal F} \subset S$ be  a collection
(maybe empty) of homotopically independent  pairwise disjoint
simple closed curves. Consider the quotient space $S  =S/\mathcal{F}$ obtained by
identifying the points belonging to the same curve in ${\mathcal
F}$. We say that $S$ is a {\it stable surface of genus $n$} and
that each element of $S$ which is the projection of a curve in
${\mathcal F}$ is a {\it node} of $S$. The connected components of the complement of the nodes are called {\it parts}.
A  {\it stable hyperbolic surface} of genus $n$    is  a stable surface $X$ of genus $n$ so that each of its parts  has a complete hyperbolic structure  of finite area (that is, each part is a hyperbolic  surface with punctures).   An {\it isometry}   between two stable hyperbolic surfaces is a homeomorphism   whose  restriction to the complement of the nodes is an isometry. 

 A {\it stable graph} is a connected  weighted graph ${\mathcal G}$ such that   each vertex with weight zero
has  degree at least three (the {\it degree} of a vertex is the number of edges coming into it, taking into account that loops count by two). 
The {\it  genus} of ${\mathcal G}$ is defined to be the genus of the underlined unweighted graph  plus the sum of all the weights, i.e., 
$$
g=\displaystyle\sum_{i=1}^{v_{{\mathcal G}}} g_i+ e_{{\mathcal G}}-v_{{\mathcal G}}+1,
$$
where $v_{{\mathcal G}}$ is the number of vertices and $e_{{\mathcal G}}$ is the number of edges and the $g_i$ is the weight of the vertex $v_i$ of ${\mathcal G}$. An isomorphism between stable graphs is a usual graph isomorphism preserving  the weights. 

To a stable hyperbolic surface   $X$  one can associate  a stable graph ${\mathcal G}(X) $  whose vertices are the parts of $X$,    the edges are the nodes, an edge connects  two vertices if the corresponding node is in the closure of the two corresponding parts, and the weight of a vertex is the genus of the corresponding part.    Notice that the genus of a stable graph ${\mathcal G}(X)$ is equal to the genus of the normalization of $X$ (i.e., the topological surface obtained by ``opening" the nodes of $X$). If $X$ is the stable surface obtained by collapsing the multicurve $\mathcal F$ in a surface $S$, we say that ${\mathcal G}(X)$ is the stable graph dual to $(S,\mathcal F)$.

 There is a bijection between stable hyperbolic surfaces up to homeomorphism and  stable graphs up to isomorphism,  \cite{miranda}. 

\subsection{Dehn-Thurston coordinates of curve systems}\label{sec:Dehn-Thurston}

We will follow the description of Dehn-Thurston coordinates given in \cite{Luo}. These coordinates parametrize the set of multicurves of a surface with boundary, $S$, up to isotopy.  Here, a {\it  multicurve} or {\it curve system} is a proper submanifold on $S$, that is, a collection of disjoint simple closed curves and simple arcs   whose endpoints are in the boundary of $S$. Parallel curves are permitted in a multicurve. In our situation, we will have an orbifold ${\mathcal O}$ instead a surface $S$, and we will consider each cone point of ${\mathcal O}$ as a boundary component. In this way, we can use Dehn-Thurston coordinates to parametrize multicurves of ${\mathcal O}$ which avoid the singular locus, except for the endpoints of their arc components.

Briefly, once fixed a pants decomposition of the surface, the Dehn-Thurston coordinates assign to each multicurve its intersection numbers with each pants curve, and the twisting numbers around each pants curve. For our purposes, the twisting numbers will not be necessary. Since these numbers need quite a few details to be defined, we will not recall their definition here. 

The first case of surface (with negative Euler characteristic) is the  pair of pants or three-holed sphere, $S_{0,3}$. For this case, the Dehn-Thurston coordinates of a multicurve $\alpha$ is the triple $(x_1,x_2,x_3)$ where $x_i$ is the  intersection number of $\alpha$ with the boundary component $\partial_i$. We need to define a {\it standard} representative of a multicurve in $S_{0,3}$, relative to a colored $H$-decomposition of $S_{0,3}$.  A {\it  colored $H$-decomposition } of the oriented surface $S_{0,3}$ is a curve system $b_1\cup b_2\cup b_3$ cutting $S_{0,3}$ into two hexagons which we paint in red and white (the curve system $b_1\cup b_2\cup b_3$ has Dehn-Thurston coordinates $(2,2,2)$). We denote by $b_i$ the arc disjoint from the boundary component $\partial_i$.  A multicurve in $S_{0,3}$ is {\it standard} if each of its components is standard. An arc $\alpha$ in $S_{0,3}$ is {\it standard} if either
  \begin{itemize}
  \item  it is contained in the red hexagon; or
  \item  if $\partial \alpha\subset \partial_i$, then: $\partial \alpha_i$ is contained in the red hexagon; $|\alpha\cap (b_1\cup b_2\cup b_3)|=|\alpha\cap (b_i\cup b_j)|=2$; and the cyclic order of $\alpha\cap \partial_i,\alpha\cap b_i,\alpha\cap b_j$ in the boundary of the red hexagon agrees with the   orientation    induced from  $S_{0,3}$. 
  \end{itemize}
Each curve system in $S_{0,3}$ has a standard representative. For instance, Figure \ref{fig:DH-coordinates} shows examples of  standard curve systems. In the figure, the red hexagon is the one above, and the orientation is given by the anticlockwise orientation in the red hexagon.

\psfrag{d1}{$\delta_1$}
\psfrag{d2}{$\delta_2$}
\psfrag{d3}{$\delta_3$}
\psfrag{b1}{$b_1$}
\psfrag{b2}{$b_2$}
\psfrag{b3}{$b_2$}
 \psfrag{(4,1,1)}{$(4,1,1)$} 
 \psfrag{(1,3,0)}{$(1,3,0)$} 
\psfrag{a}{$\alpha$}   
   \begin{figure}
 \center
 \includegraphics[height=4cm,width=10cm]{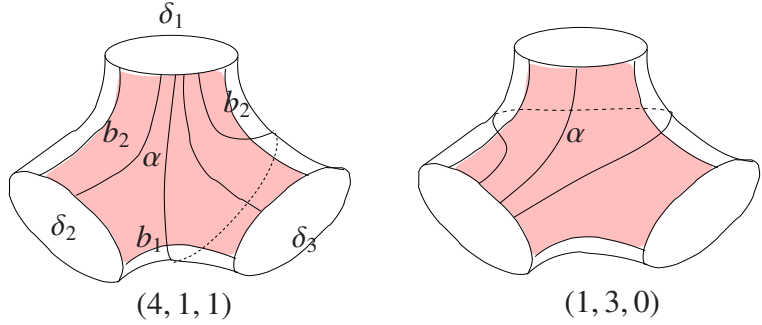}
 \caption{Standard representatives}
\label{fig:DH-coordinates}
 \end{figure}
 
To define the Dehn-Thurston coordinates for an arbitrary surface, we first need to choose a standard representative, which is done in the following way. 

Let $S$ be an oriented surface with boundary. A {\it colored $H$-decomposition} of $S$ is a pair of curve systems $(p,b)$ such that:
\begin{itemize}
\item $p,b$ are chosen in their isotopy class with the property of having minimal intersection.
\item $p$ is a pants decomposition of $S$, that is,  a maximal system of disjoint simple closed curves not homotopic among them and not homotopic to the boundary. It is well known that if $S=S_{g,r}$, i.e., it has genus $g$ and $r$ boundary components, then the number of curves in any pants decomposition is equal to $3g-3+r$. The components of $p$ are denoted $p_i, i=1,\dots, 3g-3+r$. The boundary components  of $S$  will also be called pants curves, and denoted $p_{3g+r-2},\dots, p_{3g+2r-3}$.
\item $b$ is a curve system whose intersection with any pair of pants of $S\setminus p$ is a  multicurve in $S_{0,3}$ with  coordinates $(2,2,2)$.
\item The hexagons of $S\setminus (p\cup b)$ are colored red and white with the condition that there is exactly one red hexagon in each pair of pants of $S\setminus p$   and a pair of  red and white  hexagons   never share a $p$-edge.
\end{itemize}

 An  isotopy class  curve system  $[\alpha]$ of $S=S_{g,r}$ has {\it Dehn-Thurston coordinates}  
 $$((x_1,0) , \dots,  (x_{3g+r-3},0), x_{3g+r-2},\dots x_{3g+2r-3})$$ 
 if there is a representative $\alpha$ such that  the intersection number of $\alpha$ with the pants curve $p_i$ is $x_i$ and the intersection of $\alpha$ with each pair of pants of $S\setminus p$ is a standard curve system of that pair of pants.
 
 The main result about Dehn-Thurston coordinates (restricted to the special case  we are considering, where  all the twists are zero) states that given  non-negative integers $x_1,$ $\dots ,$ $ x_{3g+2r-3}$, there is a multicurve $\alpha$ with Dehn-Thurston coordinates $((x_1,0),\dots, (x_{3g+r-3},0)$, $x_{3g+r-2}$, $\dots,$  $x_{3g+2r-3})$ if and only if whenever $p_i,p_j,p_k$ bound a pair of pants of $S\setminus p$, then $x_i+x_j+x_k$ is even.

\section{Statement of main theorem in \cite{Di-Go}}\label{sec:MainThmInDiGo}

Let $(S,G,\iota)$ be an action of the group $G$ on $S$, let $p\colon S\to {\mathcal O}$ be the associated branched covering and let  $\Phi\colon \pi_1({\mathcal O},*)\to G$ be the associated epimorphism.
Let  $\Sigma=\{{\gamma}_1,\dots, {\gamma}_k\}$ be a multicurve   in ${\mathcal O}$.  Theorem 4.1 in \cite{Di-Go} describes the stable graph ${\mathcal D}_{\tilde\Sigma}$ dual to $(S,\tilde\Sigma=p^{-1}(\Sigma))$, which here  we denote by ${\mathcal D}_{\tilde\Sigma}$.  
To  recall its statement we first need to do some notations and fix some choices.

\subsection*{Notations.}

The  complement ${\mathcal O}\setminus \cup_i {\gamma}_i$ is a union of (open) suborbifolds   ${\mathcal O}_1,\dots, {\mathcal O}_r$. We denote by $\bar {\mathcal O}_j$ the closure of ${\mathcal O}_j$ in ${\mathcal O}$. We do the following definitions
$$
\Sigma_j=\{{\gamma}\in \Sigma  \,:\, {\gamma}\subset\bar {\mathcal O}_j \}\quad \hbox{for each } j=1,\dots, r
$$ 
 $$
 \Sigma^1=\{{\gamma}\in \Sigma \,:\, \hbox{there exists a unique } j\in\{1,\dots,r\} \hbox{ with } {\gamma}\subset \bar{\mathcal O}_j \}
 $$
  $$
 \Sigma^2=\{{\gamma}\in \Sigma \,:\, \hbox{there exist different } j,j'\in\{1,\dots,r\} \hbox{ with } {\gamma}\subset \bar{\mathcal O}_j\cap \bar{\mathcal O}_{j'} \}
 $$

  We have already fixed a basepoint $*$ on ${\mathcal O}$.
Let  $X$ be a path-connected subset of $ {\mathcal O}$. 
Choose a basepoint $*_X\in X\setminus {\rm Sing}\,X$  and a path $\beta_X$ from $*$ to $*_X$ and   consider the homomorphism  $i_*\colon \pi_1(X,*_X)\to \pi_1({\mathcal O},*)$ defined by $i_*(\alpha)=\beta_X\alpha\beta_X^{-1}$. Now, define $\Phi_X=\Phi\circ i_*$. We will use these homomorphisms for the particular cases that $X={\mathcal O}_j$ (in this case $\Phi_{{\mathcal O}_j}$ will be abbreviated to $\Phi_j$) and  $X={\gamma}$, with ${\gamma}\in\Sigma$.

\subsection*{Previous choices}
We first need to choose basepoints and paths in a suitable way. 

Let $\mathcal{R}$ be the graph whose vertices are the suborbifolds ${\mathcal O}_1,\dots, {\mathcal O}_r$, whose edges are the elements of $\Sigma^2$ and two different vertices ${\mathcal O}_{j_1},{\mathcal O}_{j_2}$ are joined with the edge ${\gamma}$ if and only if  ${\gamma}\in \Sigma_{j_1}\cap\Sigma_{j_2}$.

\begin{itemize}

\item[(1)] {\it (Tree of suborbifolds $\mathcal{T}$)} We choose a    spanning tree $\mathcal{T}$ of $\mathcal R$. 
\item[(2)] {\it (Basepoints)} For each $j=1,\dots , r$, take a basepoint $*_j$ in ${\mathcal O}_j\setminus {\rm Sing}\,{\mathcal O}$. We choose one of the $*_j$ as basepoint for ${\mathcal O}$, for instance $*=*_1$. For any  ${\gamma}={\gamma}_i\in \Sigma$ choose  a basepoint  $*_{{\gamma}}\in{\gamma}\setminus{\rm Sing}\,{\mathcal O}$. 
(See Figure \ref{fig:CaminosCGamma}.)
\item[(3)] {\it (Paths $\beta_{j,{\gamma}}$)} Let$j=1,\dots ,r$.   For  each ${\gamma}\in\Sigma^2\cap \Sigma_j$ we  consider a  simple    path 
$\beta_{j,{\gamma}}$ from $*_j$ to $*_{{\gamma}}$.
For 
   each ${\gamma}\in\Sigma^1\cap \Sigma_j$  we consider two   simple paths  $\beta_{j,{\gamma}}^a$, $\beta_{j,{\gamma}}^b$
 from $*_j$ to $*_{{\gamma}}$ such that $\beta_{j,{\gamma}}^a(\beta_{j,{\gamma}}^b)^{-1}$ intersects   ${\gamma}$ exactly once and, in the case that ${\gamma}$ is an arc, $\beta_{j,{\gamma}}^a(\beta_{j,{\gamma}}^b)^{-1}$ bounds a disc which contains just one of the endpoints of ${\gamma}$  and no other cone point. 
   Moreover we choose all these paths   so that  they are disjoint except at their endpoints. 
 
 \item[(4)] {\it (Paths $\beta_{j}$ and $\beta_{{\gamma}}$)}   We will choose the paths  $\beta_j,\beta_{{\gamma}} (j=1,\dots, r, {\gamma}\in \Sigma)$ going along the paths  $\beta_{j,{\gamma}}$, as follows. We choose  $\beta_1$ to be the constant path. For $j=2,\dots, r$, let $T_j={\mathcal O}_1{\gamma}_{i_1}{\mathcal O}_{j_2}{\gamma}_{i_2}\dots {\mathcal O}_j$ be the path in the tree $\mathcal T$ from ${\mathcal O}_1$ to ${\mathcal O}_j$, given by its sequence of vertices and edges. The path $\beta_j$ is determined from $T_j$, replacing each occurrence of ${\mathcal O}_{k}{\gamma} {\mathcal O}_{k'}$ by $\beta_{ k,{\gamma} }\beta_{k',{\gamma} }^{-1}$. 
Finally, let ${\gamma}\in\Sigma$. If ${\gamma}\in \Sigma^1$, then we take $\beta_{{\gamma}}=\beta_j\beta_{j,{\gamma}}^a$. Otherwise, if ${\gamma}\in\Sigma_{j_1}\cap \Sigma_{j_2}$ with $j_1<j_2$, we   take $\beta_{{\gamma}}=\beta_{j_1}\beta_{j_1,{\gamma}}$.
 
\item[(5)] {\it (Paths $c_{{\gamma}}$)} For each ${\gamma}\in \Sigma$, we   denote by $c_{{\gamma}}$ the loop:
$$
c_{{\gamma}}=\left\{ 
\begin{matrix}
\beta_{{\gamma}} (\beta_{j,{\gamma}}^b)^{-1}\beta_j^{-1}=\beta_{j}\, \beta_{j,{\gamma}}^a(\beta_{j,{\gamma}}^b)^{-1}\beta_j^{-1} , & {\rm if} & {\gamma}\in \Sigma^1\cap \Sigma_j; \cr \cr
\beta_{{\gamma}} \, \beta_{j_2,{\gamma}}^{-1}\,\beta_{j_2}^{-1} =\beta_{j_1}\,\beta_{j_1,{\gamma}}\, \beta_{j_2,{\gamma}}^{-1}\,\beta_{j_2}^{-1} , & {\rm if} & {\gamma}\in \Sigma_{j_1}\cap \Sigma_{j_2}, j_1<j_2.
\end{matrix}
\right.
$$ 
  
\end{itemize} 
  
  \psfrag{*}{$*$}

\psfrag{*j1}{$*_{j_1}$}
\psfrag{*j2}{$*_{j_2}$}

\psfrag{*g}{$*_{\gamma}$}
\psfrag{*j}{$*_{j}$}
\psfrag{bj1}{$\beta_{j_1}$}
\psfrag{bj2}{$\beta_{j_2}$}
\psfrag{bg}{$\beta_{\gamma}$}
\psfrag{bj1g}{$\beta_{j_1,\gamma}$}
\psfrag{bj}{$\beta_{j}$}
\psfrag{bjga}{$\beta_{j,\gamma}^a$}
\psfrag{bjgb}{$\beta_{j,\gamma}^b$}

\psfrag{bj2g}{$\beta_{j_2,\gamma}$}

\psfrag{Oj1}{$\bf{\mathcal{O}_{j_1}}$}
\psfrag{Oj2}{$\bf{\mathcal{O}_{j_2}}$}
 \psfrag{D}{$\Delta$}
\psfrag{g}{\textcolor{red}{ $ \gamma$} }
\psfrag{d}{$ {\delta}$}

\psfrag{Oj}{$\bf{\mathcal{O}_{j}}$}
\psfrag{cg}{\textcolor{blue}{$c_{{\gamma}}$}}
  \begin{figure}
\center
\includegraphics[height=5.5cm,width=15.5cm]{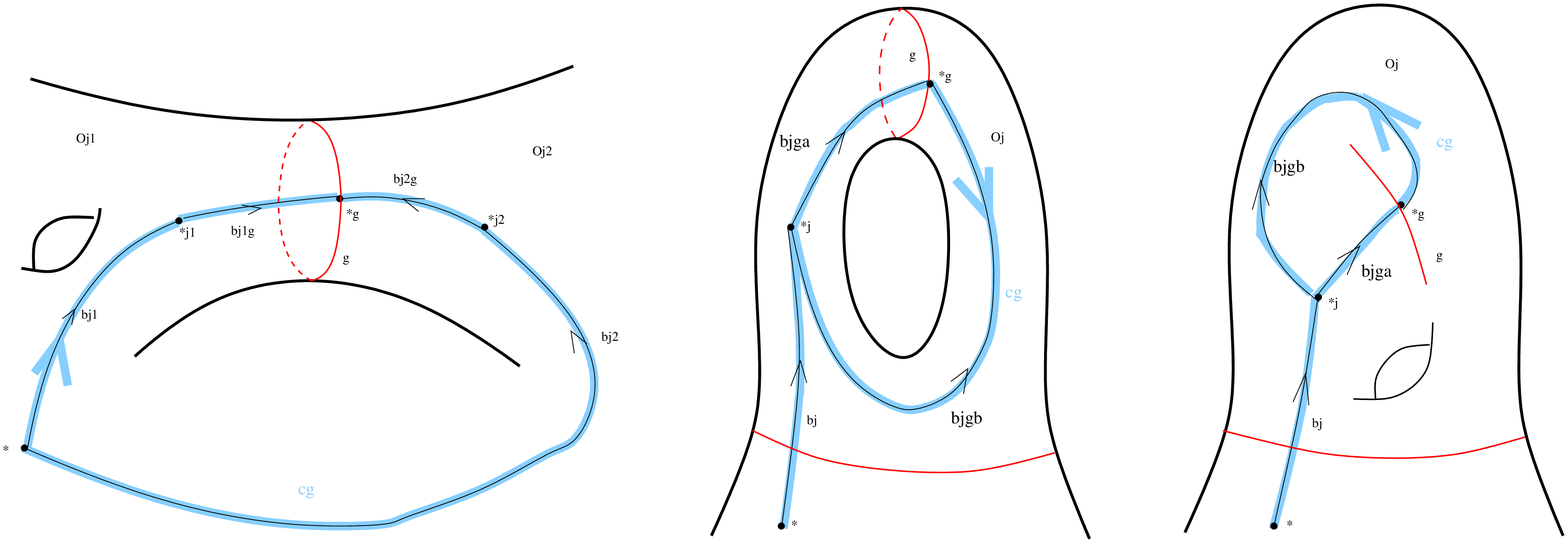}
\caption{Paths $c_{\gamma}$}
\label{fig:CaminosCGamma}
\end{figure}

 \begin{thm}\label{thm:Combinatorial}
Let $(S,G,\iota)$ be an action of the group $G$ on $S$, let $p\colon S\to {\mathcal O}$ be the associated branched covering and let  $\Phi\colon \pi_1({\mathcal O},*)\to G$ be the associated epimorphism.
\\
Let  $\Sigma=\{{\gamma}_1,\dots, {\gamma}_k\}$ be a multicurve   in ${\mathcal O}$ which decomposes ${\mathcal O}$ in the suborbifolds $\,{\mathcal O}_1,\dots, {\mathcal O}_r$,
and let  $\mathcal{D}_{\tilde\Sigma}$ be the stable graph dual to $(S,\tilde \Sigma=p^{-1}(\Sigma))$. 
  We consider the above notations and choices.  
  
   Then: 
 \begin{itemize}

 \item[(i)] (Number of vertices  of ${\mathcal D}_{\tilde\Sigma}$.) 
  The number of vertices of  ${\mathcal D}_{\tilde\Sigma}$ is equal to
 $$
 V =  
  \frac{|G|}{|{\rm Im}\, \Phi_1|} + \dots, +  \frac{|G|}{|{\rm Im}\, \Phi_r|}
 $$
The vertex of ${\mathcal D}_{\tilde\Sigma}$ are  denoted  by $V_{j,g{\rm Im}\, \Phi_j}$ (or simply by $V_{j,g}$), for $j=1,\dots, r$, $g\in G$.
  \item[(ii)]  (Number of edges of ${\mathcal D}_{\tilde\Sigma}$.)
The number of edges  of  ${\mathcal D}_{\tilde\Sigma}$ is equal to
 $$
E =  
  \frac{|G|}{|{\rm Im}\, \Phi_{{\gamma}_1}|} + \dots, +  \frac{|G|}{|{\rm Im}\, \Phi_{{\gamma}_k}|}.
 $$  
 
The edges of ${\mathcal D}_{\tilde\Sigma}$ are denoted by $E_{{\gamma},g{\rm Im}\,\Phi_{{\gamma}}}$ (or simply by $E_{{\gamma},g}$), 
   for ${\gamma}\in\Sigma$, $g\in G$.
 \item[(iii)] (Degrees of    vertices.) The degree of the vertex $V_{j,g}$ 
is equal to 
 $$
 D_{j}= |{\rm Im}\, \Phi_j| \left( \sum_{{\gamma}\in\Sigma_j\cap\Sigma^2}\frac{1}{|{\rm Im}\,\Phi_{{\gamma}}|} +2\sum_{{\gamma}\in\Sigma_j\cap\Sigma^1}\frac{1}{|{\rm Im}\,\Phi_{{\gamma}}|} \right).
  $$
 
 \item[(iv)]  (Weights of  vertices.) The weight $w^j$
 of the vertex $V_{j,g }$ is equal to   
$$\quad \quad
w^j
=1 
 -\frac{1 }{2}\left(|{\rm Im}\, \Phi_j|\, \chi({\mathcal O}_j)+ D_j\right).  
$$ 
 
\item[(v)]  (Edges connecting vertices.) Let $E_{{\gamma},g{\rm Im}\,\Phi_{{\gamma}}}$ be an edge, ${\gamma}\in \Sigma_{j_1}\cap \Sigma_{j_2}$, $j_1\leq j_2$. Then $E_{{\gamma},g }$ joins the vertices $V_{j_1, g}$ and  $V_{j_2, g\Phi(c_{{\gamma}})}$ 
(see Figure \ref{fig:CaminosCGamma}).
 
 \end{itemize}

 \end{thm}

\section{The pyramidal action of the dihedral group on $S_n$.} 
\label{sec:PyaramidalAction}
 Let $G=D_n$ be the dihedral group of order $2n$,
with presentation 
$$
D_n=\langle \rho, \sigma \;: \; \rho^n=\sigma^2=(\sigma\rho)^2=1 \rangle.
$$ 

In the sequel, the elements of the subgroup  $\langle \rho\rangle$ will be called {\it rotations},  while the remaining elements will be called {\it symmetries}.
We recall that the conjugate of a rotation $\rho^k$ is $\rho^{\pm k}$,   the conjugate of a symmetry is another symmetry, and  if $r$ is any rotation and $s$ any symmetry, then $rs=sr^{-1}$.

Let $S$ be a surface of genus $n\geq 3$, and let ${\mathcal O}$ be the orbifold of signature $(0;2,2,2,2,n)$. We denote by  $P_1, \dots, P_4$  the  cone points of order 2 and by $P_5$ the cone point of order $n$, and let $x_i$ be a loop surrounding $P_i$. Then the fundamental group of this orbifold has presentation 
 $$
 \pi_1({\mathcal O}, *)=\langle x_1,x_2,x_3,x_4,x_5 \;: \;  x_1^2=x_2^2=x_3^2=x_4^2=x_5^n=x_1x_2x_3x_4x_5=1 \rangle.
 $$
The pyramidal action of $D_n$ on $S_n$ is the action determined by the epimorphism
$\Phi\colon \pi_1({\mathcal O})\to D_n$  defined as 
  $$\Phi(x_1)=\sigma, \quad \Phi(x_2)=\Phi(x_3)=\Phi(x_4)=\rho\sigma, \quad \Phi(x_5)=\rho, $$ 
  
 Let $\mathcal{P}_n$ be the equisymmetric locus determined by the pyramidal action, i.e., the set of hyperbolic surfaces of genus $n$ such that their preserving orientation  isometry groups actions  are topologically equivalent to the pyramidal action of the dihedral group $D_{n}$ on $S_n$. We call $\mathcal{P}_n$ the {\it pyramidal equisymmetric locus}.

 Let ${\mathcal G}$ be a stable graph of   genus $n$ and let  $\mathfrak{S}({{\mathcal G}})$ be  the stratum of $\partial {\mathcal M}_n$ determined by ${\mathcal G}$. We recall that ${\mathcal G}$ is   realizable in $\partial\mathcal{P}_n$  if $\mathfrak{S}({{\mathcal G}})\cap \partial\mathcal{P}_n{\mathbb N}ot=\emptyset$.
   Precisely, in \cite{Di-Go} we proved the following (Theorems 5.1,  5.2, 5.3 and 5.4 in \cite{Di-Go}).
 
 \begin{thm}\label{thm:pyramids}
 Let ${\mathcal G}$ be a stable graph. If ${\mathcal G}$ is realizable in $\partial\P_n$, then it is  a graph in the following list.
 \begin{enumerate}
 \item ${\mathcal G}$ has one vertex and $n/m$ edges, where   $m\geq 1$ divides   $n$   (thus, the degree of the vertex is $2n/m$ and its weight is $n-n/m$); we denote this graph by ${\mathcal G}^1_{n,m}$ (see Figure ~\ref{fig:Graphs}).
 \item ${\mathcal G}$ has  two vertices with the same weight $w$ and $E$ edges joining one vertex  to the other, where $2w+E-1=n$, and where  $E=(n,k)+(n,k+1)$ for some $k=1,\dots, n$. We denote this graph by ${\mathcal G}^2_{n,k}$.
 \item  There is a number $m\geq 1$ dividing $n$ so that the graph ${\mathcal G}$ has $\frac{n}{m}+1$ vertices, one with weight 0 and degree $n$ and the others with weight 1 and degree $m$, and $n$ edges, each of them  joining the vertex of weight 0 with a different vertex;  we denote this graph by ${\mathcal G}^3_{n,m}$.
 \item  There exists $m\geq 1$ dividing $n$ and there exists $d\geq 1$ dividing $n/m$  so that the graph ${\mathcal G}$ has:
 \begin{itemize}
 \item[i)] $\frac{n}{m}+1$ vertices $V_0 ,V_1,\dots, V_{n/m}$ all of them with weight 0, $V_0$   with degree $n$ and the others with degree $m+2$;
 \item[ii)] $n+\frac{n}{m}$ edges;
 \item[iii)] for each vertex $V_i, i>0$  there are $m$ edges joining it to $V_0$;
 \item[iv)] the remaining edges join the vertices $V_i, i>0$ in cycles of length $d$.
\end{itemize} 
 We denote this graph by ${\mathcal G}^4_{n,m,d}$.
 \end{enumerate} 
Moreover, if ${\mathcal G}$ is a stable graph of   genus  $n$ and of the type  ${\mathcal G}^1_{n,m}, {\mathcal G}^2_{n,k}, {\mathcal G}^3_{n,m} $ or   is a graph of type ${\mathcal G}^4_{n,m,d}$ with $d=1,2$ or $\frac{n}{m}$, then ${\mathcal G}$ is realizable.

 \end{thm}

 \psfrag{n/m}{\small{$\frac{n}{m}$}}
     \psfrag{n-n/m}{\small{$n-\frac{n}{m}$}}       
\psfrag{*g1}{$*_{{\gamma}_1}$}
\psfrag{*g2}{$*_{{\gamma}_2}$}
\psfrag{g1}{${\gamma}_1$}
\psfrag{t}{$t$}
\psfrag{g2}{${\gamma}_2$}
\psfrag{1=rho}{ }
\psfrag{2=rho2}{ } 
\psfrag{1=1}{$t=1$ }
\psfrag{2=rho}{ }
\psfrag{t=0}{$t=0$}
\psfrag{1,rho}{ }
\psfrag{w}{\small{$w$}}
\psfrag{n,k}{\small{$(n,k)+(n,k+1)$}}
\psfrag{vd}{$\vdots$}

\psfrag{1}{\small{1}}
\psfrag{0}{\small{0}}

 \begin{figure} 
\includegraphics[height=3.0cm,width=15cm]{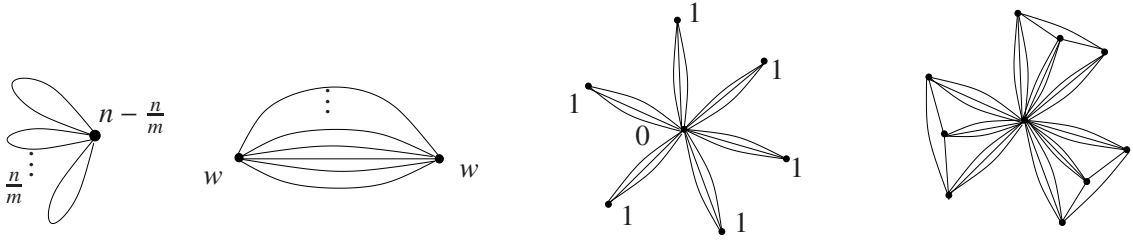}
\caption{Graphs of type ${\mathcal G}^1_{n,m}, \,{\mathcal G}^2_{n,k}, \,{\mathcal G}^3_{n,m} $ and ${\mathcal G}^4_{n,m,d}$. No number assigned to vertices means weight zero. }
\label{fig:Graphs}
\end{figure}
 
 In this paper we prove the following.
 
  \begin{thm}
  \label{thm:Grafo4}
 Let $n,m,d$ be positive integers so that $m$ divides $n$ and $d$ divides  $\frac{n}{m}$. Then the graph   ${\mathcal G}^4_{n,m,d}$ is realizable in $\partial\P_n$.
 \end{thm}

The idea for the proof is to use Dehn-Thurston coordinates, which parametrize all multicurves in ${\mathcal O}$.   Systematically applying Theorem \ref{thm:Combinatorial} to these multicurves, we   obtain the complete stratification of $\partial\mathcal{P}_n$. If we want to know whether or not a stable graph ${\mathcal G}$ is realizable in $\partial\mathcal{P}_n$, either we find  a multicurve $\Sigma$ in ${\mathcal O}$ so that ${\mathcal G}_{\Sigma}={\mathcal G}$ or we prove there is no such curve. The difficulty, of course, is that there are infinitely many isotopy classes of multicurves.

 \section{Proof of Theorem \ref{thm:Grafo4}} \label{sec:proof}

Given a graph ${\mathcal G}^4_{n,m,d}$, we need to find a multicurve $\Sigma $ such that ${\mathcal G}_{\Sigma}= {\mathcal G}^4_{n,m,d}$. 
The multicurve $\Sigma$ will consist on one arc ${\gamma}_1$ and a closed curve ${\gamma}_2$. We start by recalling from \cite{Di-Go} that, in this case, ${\mathcal D}_{\tilde\Sigma}={\mathcal G}^4_{n,m,d}$, for some $m,d$ depending on $\Sigma$ that will be defined next.

\subsection{Multicurve: one arc and  one closed curve.}
Let us  first fix the basepoint $*$ and generators $x_1$ , $\dots,$ $ x_5$ of the fundamental group of the orbifold ${\mathcal O}$ of  signature $(0;2,2,2,2,n)$.
  
Let  $\Sigma=\{{\gamma}_1,{\gamma}_2\}$  be an (admissible) multicurve in ${\mathcal O}$, where ${\gamma}_1$ is an arc and ${\gamma}_2$ is a closed curve. Then ${\mathcal O}-\Sigma={\mathcal O}_1\cup{\mathcal O}_2$, where ${\mathcal O}_1$ is an annulus with a cone point of order 2, and ${\mathcal O}_2$ is a disc with a cone point of order 2 and a cone point of order $n$. We can assume that $*\in{\mathcal O}_1$.
Consider basepoints $*_1=*, *_2,*_{{\gamma}_1},*_{{\gamma}_2}$ and paths $\beta_{1,{\gamma}_1}^a,\beta_{1,{\gamma}_1}^b,\beta_{1,{\gamma}_2}, \beta_{2,{\gamma}_2}, \beta_1,\beta_2,\beta_{{\gamma}_1},\beta_{{\gamma}_2}$ as explained in Section \ref{sec:MainThmInDiGo} (see Figure \ref{fig:Fundamental}). Notice that $c_{{\gamma}_1}=\beta_{1,{\gamma}_1}^a (\beta_{1,{\gamma}_1}^b)^{-1}$ and $c_{{\gamma}_2}$ is trivial.

 \psfrag{z}{\tiny{$z$}} 
 \psfrag{O1}{\tiny{${\mathcal O}_1$}}
 \psfrag{O2}{\tiny{${\mathcal O}_2$}}
  \psfrag{g1}{\tiny{${\gamma}_1$}}
   \psfrag{g2}{\tiny{${\gamma}_2$}}
  \psfrag{b1a}{\tiny{$\beta^a_{1,{\gamma}_1}$}}
  \psfrag{b1b}{\tiny{$\beta^b_{1,{\gamma}_1}$}}
   \psfrag{b1g2}{\tiny{$\beta_{1,{\gamma}_2}$}}
    \psfrag{b2g2}{\tiny{$\beta_{2,{\gamma}_2}$}}
\psfrag{s1}{\tiny{$*_1$}}
\psfrag{s2}{\tiny{$*_2$}}
\psfrag{sg1}{\tiny{$*_{{\gamma}_1}$}}
\psfrag{sg2}{\tiny{$*_{{\gamma}_2}$}}

\begin{figure}
\includegraphics[height=5cm,width=5cm]{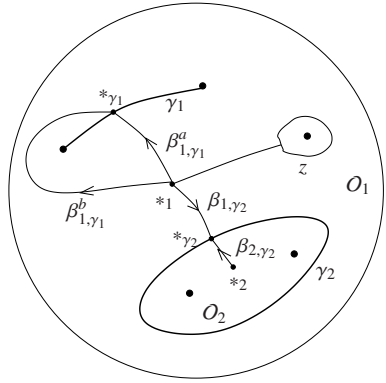}
\caption{Notations for paths. }
\label{fig:Fundamental}
\end{figure}

Notice that the (orbifold) fundamental  group $\pi_1({\gamma}_1,*_{{\gamma}_1})$ is generated by the loops ${\gamma}_1^a, {\gamma}_1^b$, each one of them going from $*_{{\gamma}_1}$ until one endpoint of ${\gamma}_1$ and coming back. For instance, we can assume that ${\gamma}_1^a$ is homotopic to $(\beta^b_{{\gamma}_1})^{-1}\beta^a_{{\gamma}_1}$. Notice that $\Phi(c_{{\gamma}_1})=\Phi_{{\gamma}_1}({\gamma}_1^a)$.

  In this situation  we have: 
\begin{enumerate}

\item[(a)]  $\pi_1({\mathcal O}_2 ,*_2)$ is  generated by a loop around the cone point of order $n$ and a loop around another cone point. Thus,  ${\rm Im}\,\Phi_2$ is generated by a conjugate of a $\rho$ (which is $\rho^{\pm 1}$) and by a conjugate of a symmetry, which is again a symmetry. Thus, 
 ${\rm Im}\,\Phi_2$ is  the whole group, and so  $|{\rm Im}\,\Phi_2|=2n$.
\item[(b)]     ${\rm Im}\,\Phi_{{\gamma}_2}$ is generated by a product of a rotation and a symmetry, which is a symmetry, so that $|{\rm Im}\,\Phi_{{\gamma}_2}|=2$.
\item[(c)]  \label{def:m}  ${\rm Im}\,\Phi_{{\gamma}_1}=\langle \Phi_{{\gamma}_1}({\gamma}_1^a), \Phi_{{\gamma}_1}({\gamma}_1^b) \rangle$ is generated by two symmetries of $D_n$, so it is isomorphic to $D_m$ for some $m=m({\gamma}_1)\geq 1$ dividing $n$. Thus, $|{\rm Im}\,\Phi_{{\gamma}_1}|=2m$. We call $S=\Phi_{{\gamma}_1}({\gamma}_1^a)$. We remark that $m$ is the order of the rotation $\Phi_{{\gamma}_1}({\gamma}_1^a{\gamma}_1^b)$.
\item[(d)]     Finally, $\pi_1({\mathcal O}_1,*_1)$ is generated by a curve ${\gamma}'_1 $  that  surrounds the arc ${\gamma}_1$,
 and by  a loop $z$ surrounding the cone point of ${\mathcal O}_1$. Notice that $\Phi_1({\gamma}_1')=\Phi_{{\gamma}_1}({\gamma}_1^a{\gamma}_1^b)$ and that  $\Phi_1(z)=\Phi(z)$ is a symmetry, which we denote by $s$. Thus, ${\rm Im}\,\Phi_{ 1}$ is also a subgroup isomorphic to $D_m$.  
We remark  that ${\rm Im}\,\Phi_{{\gamma}_1}$ and ${\rm Im}\,\Phi_1$ are not necessarily equal, but share the subgroup $C_m$  of index 2 containing all their  rotations.  
 \item[(e)]  \label{def:d} We define $R=Ss$, and let $d=d({\gamma}_1,{\gamma}_2)$ be the minimum number such that $R^d\in {\rm Im}\,\Phi_1$. If we call $C_n$ to the subgroup of rotations of $D_n$ and $C_m$ to the subgroup of rotations of ${\rm Im}\,\Phi_1$, we have that $d$ is the order of the coset $R\cdot C_m$ in $C_n/C_m$. In particular, $d$ divides $n/m$.
\end{enumerate}

 \noindent

Using the above information, we will next apply  Theorem \ref{thm:Combinatorial}  to determine the graph ${\mathcal D}_{\tilde \Sigma}$.

\begin{prop}\label{prop:Grafo4} Let $\Sigma=\{{\gamma}_1,{\gamma}_2\}$ with ${\gamma}_1$ an arc and ${\gamma}_2$ a closed curve. Let $m=m({\gamma}_1)$ as defined in (c) above, and let $d=d({\gamma}_1,{\gamma}_2)$ as defined in (e) above. Then ${\mathcal D}_{\tilde\Sigma}={\mathcal G}^4_{n,m,d}$.
\end{prop}
 
This proposition corresponds to part (a) of Theorem 5.4 in \cite{Di-Go}. We sketch here its proof and refer to \cite{Di-Go}   for details.

\begin{proof} We apply Theorem \ref{thm:Combinatorial} to our situation. Part (i) of this theorem and (a), (d) above imply that the number of vertices of ${\mathcal D}_{\tilde \Sigma}$ is equal to 
$1+\frac{n}{m}$. Notice that all the  vertices $V_{2,g{\rm Im}\,\Phi_{2}}$ are equal, and we will denote $V_0$ this vertex; while there are $\frac{n}{m}$ different vertices of the form $V_{1,g{\rm Im}\,\Phi_{1}}$, $g\in G$.

Theorem \ref{thm:Combinatorial}(ii) and (b), (c) above imply that the number of edges of ${\mathcal D}_{\tilde \Sigma}$ is equal to $n+\frac{n}{m}$. There are $n$ different edges of the form $E_{{\gamma}_2,g{\rm Im}\,\Phi_{{\gamma}_2}}$ and there are $\frac{n}{m}$ different edges of the form $E_{{\gamma}_1,g{\rm Im}\,\Phi_{{\gamma}_1}}$.

The degrees and weights of the vertices are readily computed from parts (iii) and (iv) of Theorem \ref{thm:Combinatorial}.

Finally, using Theorem \ref{thm:Combinatorial}(v), we study the connectivity among  vertices and vertices. Since ${\gamma}_2$ is in   $\bar {\mathcal O}_1\cap \bar{\mathcal O}_2$, and since $c_{{\gamma}_2}$ is trivial, we have that  $E_{{\gamma}_2,g }$  connects   $V_0=V_{1,g }$ with   $V_{2,g }$. Thus, each vertex $V_{2,g{\rm Im}\,\Phi_2 }$ is connected to $V_0$ with  $\frac{n}{m}$ edges.

On the other hand,   ${\gamma}_1$ is only contained in $\bar {\mathcal O}_1$, and recall that  $\Phi(c_{{\gamma}_1})=\Phi_{{\gamma}_1}({\gamma}_1^a)=S$. Then, for any $g\in G$, the edge  $E_{{\gamma}_1,g{\rm Im}\,\Phi_{{\gamma}_1}}$ joins $V_{1,g{\rm Im}\,\Phi_1}$ to $V_{1,gS{\rm Im}\,\Phi_1}$.
Recall from (d) above the notation $s=\Phi(z)$, where $z$ is a path in ${\mathcal O}_1$. Thus, $s\in {\rm Im}\,\Phi_1$ and $V_{1,gS{\rm Im}\,\Phi_1}=V_{1,gSs{\rm Im}\,\Phi_1}=V_{1,gR{\rm Im}\,\Phi_1}$, where $R=Ss$. In the same way,   
 the edge $E_{{\gamma}_1,gR{\rm Im}\,\Phi_{{\gamma}_1}}$ joins $V_{1,gR{\rm Im}\,\Phi_1}$ to $V_{1,gR^2{\rm Im}\,\Phi_1}$, and so on. We will obtain a cycle of $k $ edges when the vertex $V_{1,gR^k{\rm Im}\,\Phi_1}$ be equal to the first one $V_{1,g{\rm Im}\,\Phi_1}$, i.e., when $R^k\in {\rm Im}\,\Phi_1$. The smallest $k$ satisfying this is the $d=d({\gamma}_1,{\gamma}_2)$ defined in (e) above. Thus, the $\frac{n}{m}$ edges of the form $E_{{\gamma}_1,g}$ split the vertices $V_{1,g}$ in $\frac{n/m}{d}$ cycles of length $d$.
\end{proof}
 
 \subsection{\bf Proof of Theorem \ref{thm:Grafo4}}
 Let $n,m,d$ positive integers with $n\geq 3$,  $m|n$ and $d|\frac{n}{m}$. By Lemma \ref{lem:arc-g1} below, there is an arc ${\gamma}_1$ in ${\mathcal O}$ joining cone points of order 2 such that $m({\gamma}_1)=m$.
 Next, by Lemma \ref{lem:curve-g2} below, there is an admissible closed curve in ${\mathcal O}-{\gamma}_1$ (i.e., bounds a  disc containing the cone point of order $n$ and one cone point of order 2) such that $d({\gamma}_1,{\gamma}_2)=d$. Then, by Proposition \ref{prop:Grafo4}, we have that ${\mathcal D}_{\tilde \Sigma}={\mathcal G}^4_{n,m,d}$, so that this graph is realizable in $\partial {\mathcal P}_n$. Thus, the proof of  Theorem \ref{thm:Grafo4} is reduced to the proof of the two lemmas below.

\begin{lemma}\label{lem:arc-g1}
Let $n\geq 3$, let $m$  be an integer number  dividing $n$ and consider $k=\frac{n}{m}$. Then there exists an arc ${\gamma}_1$ in ${\mathcal O}$ joining cone points of order 2 such that  $\Phi( \beta_{{\gamma}_1}  {\gamma}_1^a  {\gamma}_1^b \beta_{{\gamma}_1}^{-1})=\rho^k$, thus, a rotation of order $m$. Thus $m({\gamma}_1)=m$ (recall the definition of $m({\gamma}_1)$ in (c) above).
\end{lemma} 

This lemma is Lemma 5.1 of \cite{Di-Go}.  For the sake of completeness, we prove it here, providing the Dehn-Thurston coordinates of the arc ${\gamma}_1$.

\psfrag{p1}{\tiny{$p_1$}}
\psfrag{p2}{\tiny{$p_2$}}
\psfrag{p3}{\tiny{$p_3$}}
\psfrag{p4}{\tiny{$p_4$}}
\psfrag{p5}{\tiny{$p_5$}}
\psfrag{q1}{\tiny{$q_1$}}
\psfrag{q2}{\tiny{$q_2$}}
\psfrag{s}{\tiny{$*$}}
\psfrag{(left)}{{((8,0),(0,0),0,1,1,0,0)}}
\psfrag{(right)}{{((7,0),(0,0),1,1,0,0,0)}}

 \begin{figure} 
\includegraphics[height=6cm,width=13cm]{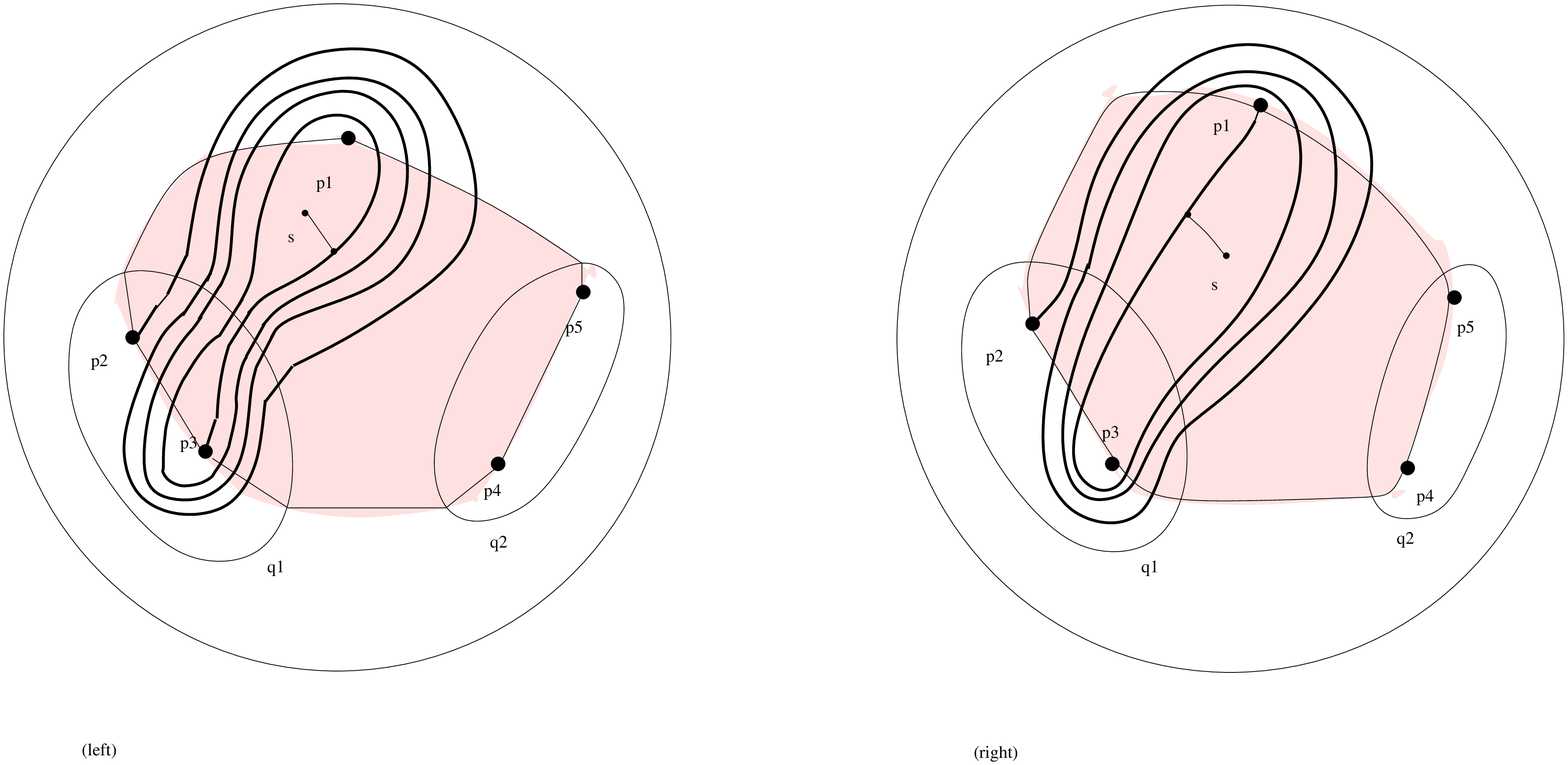}
\caption{Arc ${\gamma}_1$}\label{fig:DT-lemma1}
\end{figure}

\begin{proof}
 Actually, we will prove that for any   integer $x$ there is an arc ${\gamma}_1$ satisfying  \break
  $\Phi( \beta_{{\gamma}_1}  {\gamma}_1^a  {\gamma}_1^b \beta_{{\gamma}_1}^{-1}) = \rho^x$.
 Let us  consider an $H$-decomposition on ${\mathcal O}$, where the cone points are considered as boundary components, that is, ${\mathcal O}$ is a sphere with five boundary components $p_1,\dots, p_5$, where $p_1,\dots p_4$ correspond to the cone points of order 2 and  $p_5$ corresponds to the cone point of order $n$. We consider  two   pants curves $q_1,q_2$ such that $q_1$ bounds a disc containing the cone points $p_2,p_3$ and $q_2$ bounds a disc containing the cone points $p_4,p_5$.
  Consider arcs  $b_i$ cutting the pants into hexagons as in the Figure \ref{fig:DT-lemma1} (the hexagons are the shaded regions).  
  
  Now we consider the curve system with Dehn-Thurston coordinates  
  $$
 ( (x,0),(0,0),0, 1,1,0,0), \quad  x=2\ell
  $$ with respect to the $H$-decomposition given by $(\{q_1,q_2,p_1,\dots,p_5\},\{b_i\})$  (see the left hand side part of Figure \ref{fig:DT-lemma1} for an example with $x=8$).
  We can easily check that this curve system consists of  one arc ${\gamma}_1$. Setting the basepoints as in the figure, and choosing ${\gamma}_1^a$ the half-arc going to $p_3$,   we have that 
  $$
  \Phi(\beta_{{\gamma}_1}{\gamma}_1^a\beta_{{\gamma}_1}^{-1})= \rho\sigma, 
  $$
$$  
   \Phi(\beta_{{\gamma}_1}{\gamma}_1^b\beta_{{\gamma}_1}^{-1})= ((\sigma\rho\sigma)^{\ell-1}\sigma) \rho\sigma ((\sigma\rho\sigma)^{\ell-1}\sigma)^{-1} =\rho^{-\ell+1}\sigma\rho\sigma\sigma\rho^{\ell-1}= \rho^{-2\ell+1}\sigma  =  \rho^{-x+1}\sigma .
  $$
  Thus, we have that 
  $
    \Phi(\beta_{{\gamma}_1}{\gamma}_1^a{\gamma}_1^b\beta_{{\gamma}_1}^{-1})= \rho^{x}
  $.

On the other hand, if we consider the Dehn-Thurston coordinates  
$$
 ( (x,0),(0,0),1,1,0,0,0), \quad  x=2\ell+1
  $$ 
  with respect to the  same $H$-decomposition  (see the right hand side part of Figure \ref{fig:DT-lemma1} for an example with $x=7$), then 
 again the  curve system consists of  one arc ${\gamma}_1$. Setting the basepoints as in the figure, and choosing ${\gamma}_1^b$ the half-arc going to $p_1$,   we have that 
  $$
  \Phi(\beta_{{\gamma}_1}{\gamma}_1^b\beta_{{\gamma}_1}^{-1})= \sigma, 
  $$
$$  
   \Phi(\beta_{{\gamma}_1}{\gamma}_1^a\beta_{{\gamma}_1}^{-1})= (\rho\sigma\sigma)^{\ell} \rho\sigma (\rho\sigma\sigma)^{-\ell}
    =\rho^{2\ell+1}\sigma  =  \rho^{x}\sigma .
  $$
  Thus, 
  $
    \Phi(\beta_{{\gamma}_1}{\gamma}_1^a{\gamma}_1^b\beta_{{\gamma}_1}^{-1})= \rho^{x}
  $.
\end{proof}

\begin{lemma}\label{lem:curve-g2} Let $n,m,d$ 
positive integers with $n\geq 3$, $m$ dividing $n$, $d$ dividing $\frac{n}{m}$, and let ${\gamma}_1$ be an arc joining cone points of order 2 with $m({\gamma}_1)=m$. Then there exists an admissible closed curve ${\gamma}_2$ in ${\mathcal O}-{\gamma}_1$  such that $d({\gamma}_1,{\gamma}_2)=d$ (recall the definition of $d({\gamma}_1,{\gamma}_2)$ from (e) above).
\end{lemma}

\begin{proof}
The curve ${\gamma}_2$   should be a simple closed curve disjoint from ${\gamma}_1$, and  that bounds a disc containing the cone point of order $n$ and another cone point. To find this curve, we will look for a simple  arc ${\gamma}_2'$ in ${\mathcal O}\setminus {\gamma}_1$ joining the cone point of order $n$ and another cone point.   The curve ${\gamma}_2$ will be then a simple closed curve surrounding ${\gamma}_2'$. 

 Let us put  $k=\frac{n}{m} $ and  let ${\gamma}_1$ be an arc in ${\mathcal O}$  given in lemma \ref{lem:arc-g1}
Let ${\mathcal O}'$ be the orbifold obtained by  cutting ${\mathcal O} $ along ${\gamma}_1$.
 We  consider an $H$-decomposition on ${\mathcal O}'$, where the cone points of ${\mathcal O}'$ are considered as boundary components; that is, ${\mathcal O}'$ is a sphere with four boundary components $p_1',\dots, p_4'$, where $p_1',p_2'$ correspond to the cone points of order 2, $p_3'$ corresponds to the cone point of order $n$ and $p_4'$ corresponds to ${\gamma}_1$. We consider one more pants curve $q$
   separating ${\gamma}_1$ and $p_2'$ from $p_1',p_3'$. Consider arcs  $b_1,\dots, b_4$ joining cyclically ${\gamma}_1$ and the three cone points. In this way we obtain an $H$-decomposition in four hexagons (see  Figure \ref{fig:g2}, where the red hexagons are the two of them depicted in the upper part of ${\mathcal O}'$). 
   
Now, for any integer $x\geq 0$ we consider the multicurve ${\gamma}'_2(x)$, defined by its 
 Dehn-Thurston coordinates with respect 
to the pants curves $q,p_1',\dots, p_4'$ and the $H$-decomposition described above. If $x$ is odd, $x=2\ell+1$, then the  Dehn-Thurston coordinates of ${\gamma}'_2(x)$  are $ ((x,0),0,1,1,0 )$; if $x$  is even, $x=2\ell$, then the Dehn-Thurston coordinates of ${\gamma}'_2(x)$  are $((x,0),1,0,1,0 )$ (see Figure \ref{fig:g2}).
 We can easily check that ${\gamma}_2'(x)$ is a  single arc for any $x$; for instance, the upper right part of Figure \ref{fig:g2} shows ${\gamma}_2'(7)$ and the lower right part of this figure shows ${\gamma}_2'(6)$.  We will prove that there exists a value $x_0$  of $x$  so that ${\gamma}'_2= {\gamma}_2'(x_0)$  is the desired arc, i.e., $d({\gamma}_1,{\gamma}_2)=d$.

In order to do that, we need to compute the order of the rotation $R=\Phi( \beta_{{\gamma}_1} {\gamma}_1^a\beta_{{\gamma}_1}^{-1} z )$, where we recall that $z$ is a loop in the annulus ${\mathcal O}-({\gamma}_1\cap{\gamma}_2)$ surrounding the cone point in it. In Figure \ref{fig:g2} we have chosen such a loop $z$ for any $x\geq 0$. Of course, the paths $z$ and $R$ will  depend on $x$, but, to simplify notation, we  will denote them by $z=z(x)$ and $R=R(x)$.

We can assume that the basepoint $*$ is as in Figure \ref{fig:g2}. Let ${\gamma}'_1$ be  a loop in $\pi_1({\mathcal O}',*)$ surrounding ${\gamma}_1$ and 
oriented in such a way that  $\Phi_1({\gamma}'_1)=\rho^k$ (so ${\gamma}'_1$ is 
 homotopic to $\beta_{{\gamma}_1}  {\gamma}_1^a  {\gamma}_1^b \beta_{{\gamma}_1}^{-1}$).
We choose loops $z_1,z_2,z_3$ in $\pi_1({\mathcal O}',*)$ surrounding $p_1',p_2',p_3'$ and so that $z_1{\gamma}'_1z_2$ is homotopic to $z_3$.
 Since $z_i, i=1,2$ surrounds a cone point  of order 2, we have that $\Phi(z_i)$ is conjugate to a symmetry, and hence it is a symmetry, that we call $\sigma_i$. On the other hand, $\Phi(z_3)$ is conjugate to $\rho^{\pm 1}$, so $\Phi(z_3)=\rho^{\epsilon }$, with $\epsilon=\pm 1$. 
Since $z_1{\gamma}'_1z_2$ and  $z_3$ are homotopic, then we have that  $\sigma_1\rho^k\sigma_2=\rho^{\epsilon}$, and this implies that $\sigma_1\sigma_2=\rho^{k+\epsilon}$.

\psfrag{p1}{\tiny{$p_1$}}
\psfrag{p2}{\tiny{$p_2$}}
\psfrag{p3}{\tiny{$p_3$}}
\psfrag{g1}{\tiny{${\gamma}_1=p_4$}}
\psfrag{g2}{\tiny{${\gamma}_2'$}}
\psfrag{d}{\tiny{${\gamma}_2'(x)$}}
\psfrag{d''}{\tiny{${\gamma}_2'(x)$}}
\psfrag{q}{\tiny{$q$}}
\psfrag{s}{\tiny{$*$}}
\psfrag{z1}{\tiny{$z_1$}}
\psfrag{z2}{\tiny{$z_2$}}
\psfrag{z3}{\tiny{$z_3$}}
\psfrag{g'1}{\tiny{${\gamma}'_1$}}
\psfrag{O'}{\tiny{$\mathcal{O}'$}}
\psfrag{z}{\tiny{\textcolor{blue}{$z$}}}

\psfrag{x=par}{  $x =2\ell, \, \ell\geq 1$  }
\psfrag{x=impar}{  $x =2\ell+1, \, \ell\geq 0$  }
\psfrag{x=0}{  $x =0$  }

 \begin{figure}
\includegraphics[height=13cm,width=13cm]{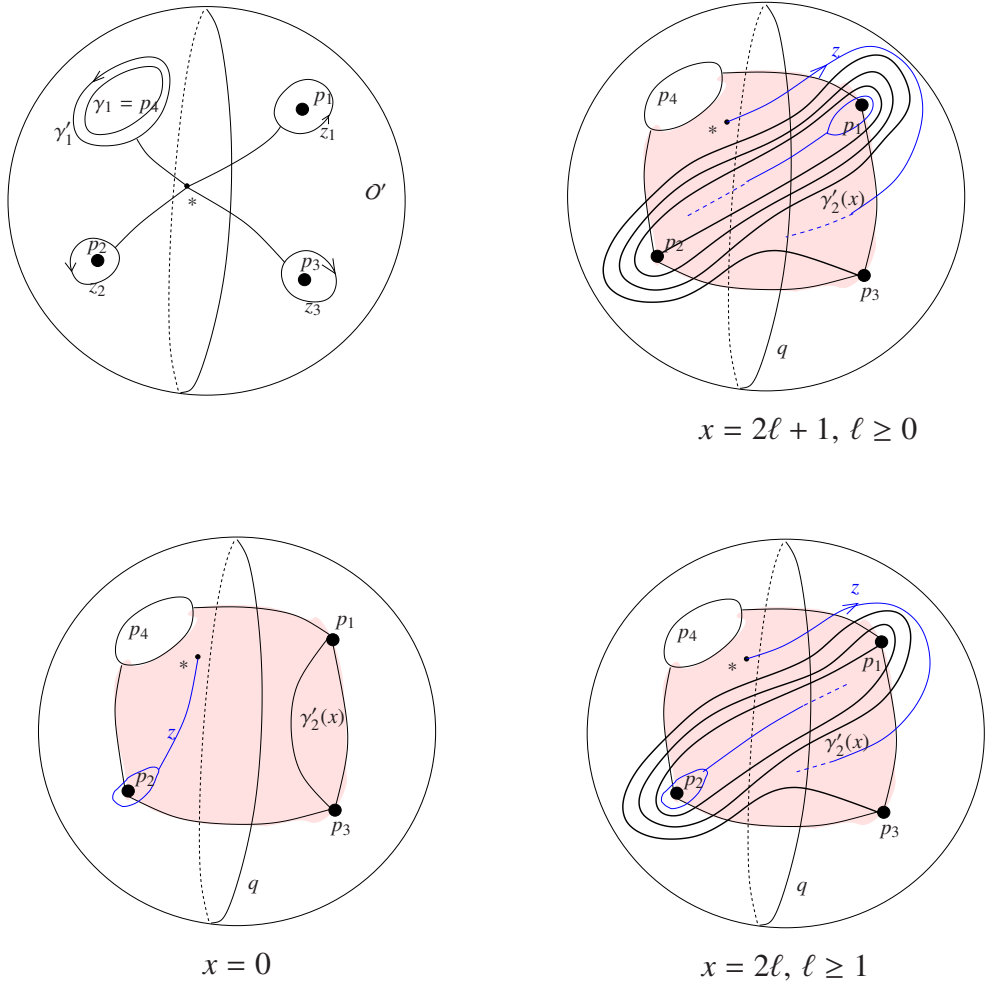}
\caption{\small{Finding the curve ${\gamma}_2$: it will be a closed curve surrounding the arc ${\gamma}_2'(x)$,  for some $x$.}} \label{fig:g2}
\end{figure} 
 
Reading $z$ from    Figure \ref{fig:g2} and taking into account that $\sigma_1,\sigma_2$ have order 2, we have (for $x\geq 0$):
\begin{itemize}
\item[i)] For $x = 2\ell+1,\,\ell \geq 0$, we have
\begin{eqnarray*}
\Phi(z)&=&
 (\sigma_1\sigma_2)^{\ell}\sigma_1     (\sigma_1\sigma_2)^{-\ell }
=(\sigma_1\sigma_2)^{2\ell }\sigma_1 =
  \rho^{ (k+\epsilon)2\ell}  \sigma_1
 = \rho^{ (k+\epsilon)(x-1)}\sigma_1.
  \end{eqnarray*}
\item[ii)] For $x=0$, $\Phi(z)=\sigma_2=\sigma_1\rho^{(k+\epsilon)}=\rho^{-(k+\epsilon)}\sigma_1$.
\item[iii)] For $x = 2\ell,\,x>0$,  we have 
\begin{eqnarray*}
\Phi(z)&=&
\left((\sigma_1\sigma_2)^{\ell-1}\sigma_1 \right) \sigma_2 \left((\sigma_1\sigma_2)^{\ell-1}\sigma_1 \right)^{-1}
=(\sigma_1\sigma_2)^{\ell-1}\sigma_1\sigma_2\sigma_1(\sigma_1\sigma_2)^{-\ell+1}
\\
&=&(\sigma_1\sigma_2)^{2\ell-1}\sigma_1
 =\rho^{ (k+\epsilon)(2\ell-1)}\sigma_1
 = \rho^{ (k+\epsilon)(x-1)}\sigma_1.
  \end{eqnarray*}
\end{itemize}
In conclusion, for any $x\geq 0$, we have that $\Phi(z)=\rho^{ (k+\epsilon)(x-1)}\sigma_1$.
  
 On the other hand,   the rotation $R$ is:
 $$
 R=\Phi(\beta_{{\gamma}_1}{\gamma}_1^a\beta_{{\gamma}_1}^{-1} z) 
 =S \Phi(z)
  =S\rho^{ (k+\epsilon)(x-1)}\sigma_1=\rho^{- (k+\epsilon)(x-1)}S\sigma_1
  = \rho^{- (k+\epsilon)(x-1)+t},
 $$
 where $t$ is such that  $S\sigma_1=\rho^t$.
 
Recall that   $k=\frac{n}{m}$ and  put  $d'=\frac{k}{d}$. We claim that, if $k>1$, then there is some $x$ such that 
 $$
 R=  \rho^{- (k+\epsilon)(x-1)+t}  =\rho^{d'\,({\rm mod}\, k)},
 $$

 that is,  the congruence equation
 $
 - (k+\epsilon)(x-1)+t={d'\,({\rm mod}\, k)} 
$ has a solution. 

Indeed, if $k>1$, then  $k$ and $k+\epsilon$ are coprime and the inverse of $(k+\epsilon) ({\rm mod} \, k)$ exists. Thus, the above equation has solution
\begin{equation}\label{eq:Sol-x}
  x={- (k+\epsilon)^{-1}(d'-t)+1\,({\rm mod}\, k)}.
 \end{equation} 
Now, taking $x_0$ a positive solution of  (\ref{eq:Sol-x}), and taking 
${\gamma}'_2={\gamma}'_2(x_0)$, 
we have that  the order of $R\cdot C_m$ in $C_n/C_m$ is equal to $\frac{n/m}{(n/m, d')}=\frac{k}{(k,d')}=\frac{k}{d'}=d$. Therefore, ${\gamma}'_2$ is the desired curve.  

For the remaining case, $k=1$, notice that $m=n$, so the only possibility for $d$ dividing $\frac{n}{m}$  is $d=1$. In this case,   let us see that  any admissible  curve $  {\gamma}_2 $    in ${\mathcal O}-{\gamma}_1$ satisfies $d({\gamma}_1,{\gamma}_2)=1$.  But this is immediate since ${\rm Im}\,\Phi_1$ is the whole group $G$ and so $R\in {\rm Im}\, \Phi_1$. 
\end{proof}

\end{document}